\documentstyle[12pt]{article}
\begin{document}

\newcommand\comp{\circ}
\newtheorem{theorem}{Theorem}[section]
\newtheorem{example}[theorem]{Example}
\newtheorem{conjecture}[theorem]{Conjecture}
\newtheorem{examples}[theorem]{Examples}
\newtheorem{proposition}{Proposition}[section]
\newtheorem{fact}[theorem]{Fact}
\newtheorem{problem}[theorem]{Open problem}

\newcommand{\I}{\mbox{{\bf I}}}
\newcommand{\A}{\mbox{{\bf Ax \ }}}
\newcommand{\qed}{\hspace*{\fill}\rule{1 ex}{1.5 ex}\\}

\newtheorem{corollary}[theorem]{Corollary}
\newtheorem{definition}[theorem]{Definition}
\newtheorem{remark}[theorem]{Remark}
\newtheorem{lemma}[theorem]{Lemma}
\newtheorem{claim}{Claim}[theorem]
\newtheorem{construction}[theorem]{Construction}
\newenvironment{proof}{\begin{trivlist}\item[]{\bf
Proof.}}{\qed\end{trivlist}}

\newcommand{\comment}[1]{\typeout{swallowing comment}}

\hyphenation{equi-v-al-ent dimen-s-ion mon-ad-ic}

\title{On Induced Subgraphs of Finite Graphs not Containing Large Empty and Complete
Subgraphs}
\author{G\'abor S\'agi\thanks{Supported by Hungarian
National Foundation for Scientific Research grant 113047.
}}


\maketitle
\bigskip

\begin{abstract}
In their celebrated paper \cite{ErdosHajnal}, Erd\H{o}s and
Hajnal asked the following: is it true, that for any finite graph
${\cal H}$ there exists a constant $c({\cal H})$ such that for
any finite graph ${\cal G}$, if ${\cal G}$ does not contain
complete or empty induced subgraphs of size at least $|V({\cal
G})|^{c({\cal H})}$, then ${\cal H}$ can be isomorphically
embedded into ${\cal G}$ ? The positive answer has become known as
the Erd\H{o}s-Hajnal conjecture.\\
\indent In Theorem \ref{mainthm} of the present paper we settle
this conjecture in the affirmative. To do so, we are studying
here the fine structure of ultraproducts of finite sets, so our
investigations have a model theoretic character. \\
\\
{\em 2010 Mathematical Subject Classification:} 03C20, 03C13, 05C55, 28E05.  \\
{\em Keywords and Phrases:} Ultraproducts, Erd\H{o}s-Hajnal
Property, Ramsey Theory of Finite Graphs.
\end{abstract}

\section{Introduction}
\label{intro} \ \indent In their celebrated paper
\cite{ErdosHajnal}, Erd\H{o}s and Hajnal asked the following
question: is it true, that for any finite graph ${\cal H}$ there
exists a constant $c({\cal H})$ such that for any finite graph
${\cal G}$, if ${\cal G}$ does not contain complete or empty
induced subgraphs of size at least $|V({\cal G})|^{c({\cal H})}$,
then ${\cal H}$ can be isomorphically embedded into ${\cal G}$ ?
The positive answer has
become known as the Erd\H{o}s-Hajnal conjecture. \\
\indent It is known to be true for a few special cases. In
\cite{ErdosHajnal} Erd\H{o}s and Hajnal proved a somewhat weaker
general result: for any finite graph ${\cal H}$ there exists a
constant $c({\cal H})$ such that for any finite graph ${\cal G}$,
if ${\cal G}$ does not contain complete or empty induced
subgraphs of size at least $e^{c({\cal H})\sqrt{log_{2}|V({\cal
G})|}}$, then ${\cal H}$ can be isomorphically embedded into
${\cal G}$. It is also known from \cite{ErdosHajnalPach}, that if
${\cal G}$ does not contain complete, or empty induced bipartite
subgraphs with both parts of size polynomial in $|V({\cal G})|$,
then ${\cal H}$ can be embedded into ${\cal G}$. In \cite{CFS} it
was shown, that for $k \geq 4$, the class of $k$-uniform
hypergraphs does not satisfy the natural generalization of the
Erd\H{o}s-Hajnal Conjecture. Further interesting related results
can be found in \cite{FoxSudakov}. For a comprehensive survey we
refer to \cite{Chud} and the references therein. \\
\indent In Theorem \ref{mainthm} of the present paper we settle
the original conjecture in the affirmative. The rest of the paper
is devoted to work out the technical details we need to prove
Theorem \ref{mainthm}, which is the main result of the paper. We
are studying here the fine structure of ultraproducts of finite
sets, consequently, our investigations have a model theoretic
character. Our methods have been inspired by nonstandard measure
theory; however, our presentation does not refer to any purely
measure theoretic method or result. We note, that utilizing a
more or less same method, several results (in rather different
areas) have been obtained recently. In this respect we refer to
\cite{SagiGyenis},
\cite{Towsner}, \cite{Malliaris} and \cite{Malliaris2}. \\
\indent The rest of the paper is organized as follows. In Section
\ref{prel} we are recalling the notions and methods we will need
later. This section has a survey character in order to keep the
present paper self contained. In this section there are no new
results. In Section \ref{proofs} we present all the proofs we
need to derive Theorem \ref{mainthm}, the main result of the
paper. As we mentioned, Theorem \ref{mainthm} settles the
Erd\H{o}s-Hajnal conjecture affirmatively. Finally, in Section
\ref{conc} we present some problems which remained open. Some of
these problems can be regarded as certain generalizations of the
Erd\H{o}s-Hajnal conjecture, some other problems are related to
the methods and results presented in Section \ref{proofs}.

\section{Preliminaries}
\label{prel}
\ \indent We start by summing up our system of
notation, which is mostly
standard. After this, we recall the notions and results we need to establish the proof of our main theorem.\\
\indent Throughout $\omega$ denotes the set of natural numbers
and for every $n \in \omega$ we have $ n = \{ 0,1,...,n-1 \}$.
Let $A$ and $B$ be sets. Then ${}^{A}B$ denotes the set of
functions whose domain is $A$ and whose range is a subset of $B$.
Moreover, ${}^{< \omega}A$ is defined to be ${}^{< \omega}A =
\cup_{n \in \omega} {}^{n}A$. In addition, $ |A| $ denotes the
cardinality of $A$; if $\kappa$ is a cardinal then $[A]^{\kappa}$
denotes the set of subsets of $A$ which are of cardinality
$\kappa$ and ${\cal P}(A)$ denotes the power set of $A$, that is,
${\cal P}(A)$ consists of all subsets of $A$. \\
\indent $\Re$ denotes the set of real numbers; $\Re^{+}$ denotes
the set of positive real numbers. In addition, ''$log$'' denotes
the logarithm function of base $2$.
Throughout we use
function composition in such a way that the rightmost factor acts
first. That is, for functions
$f,g$ we define $f \comp g (x) = f(g(x))$. \\
\indent If $I$ is a set, ${\cal A}_{i}$ is a structure for all
$i\in I$ and ${\cal F} \subseteq {\cal P}(I)$ is an ultrafilter,
then $\prod_{i \in I} {\cal A}_{i}/{\cal F}$ denotes the
ultraproduct of the ${\cal A}_{i}$ modulo ${\cal F}$. \\
\\
{\bf Graphs.} By a graph we mean a simple, undirected graph ${\cal
G} = \langle V({\cal G}), E({\cal G}) \rangle$, where $V({\cal
G})$ is the set of vertices of ${\cal G}$ (which may be finite or
infinite) and $E({\cal G})$ is the set of edges of ${\cal G}$.
Let $X \subseteq V({\cal G})$. Then $\Gamma(X)$ denotes the set
of vertices of ${\cal G}$ connected to some elements of $X$:
$$\Gamma(X) = \{a \in V({\cal G}): (\exists b \in X) (\langle a,b
\rangle \in E({\cal G})) \}.$$ If $a \in V({\cal G})$ then
instead of $\Gamma(\{a\})$ we will simply write $\Gamma(a)$. \\
\\
{\bf Ultratopologies.} Now we recall some parts of \cite{cl1a}
and \cite{ut}; for further related results we also refer to
\cite{sash}.

\begin{definition}
Let $\langle A_{i}: i \in I \rangle$ be a sequence of arbitrary
sets, let $k \in \omega$ and let ${\cal F}$ be an ultrafilter
over $I$. A $k$-ary relation $X \subseteq {}^{k}(\Pi_{i \in I}
A_{i} /{\cal F})$ is defined to be {\em decomposable} iff for each
$i \in I$, there are $X_{i} \subseteq {}^{k}A_{i}$ such that $X =
\Pi_{i \in I} X_{i} /{\cal F}$.
\end{definition}

Decomposable relations may be characterized in terms of some
topologies defined naturally on ultraproducts.

\begin{definition}
Let $k \in \omega$. Then a function $\hat{\ }: {}^{k}(\Pi_{i \in
I} {\cal A}_{i} /{\cal F})\rightarrow {}^{k}(\Pi_{i \in I} {\cal
A}_{i})$ is defined to be a $k$-dimensional choice function iff
for all $\bar{a} \in {}^{k}(\Pi_{i \in I} {\cal A}_{i}
/{\cal F})$ we have $\bar{a} = \hat{\bar{a}}/{\cal F}$. \\
\indent For a given choice function $\hat{\ }$, $\bar{a} \in
{}^{k}(\Pi_{i \in I} {\cal A}_{i} /{\cal F})$ and $X \subseteq
{}^{k}(\Pi_{i \in I} {\cal A}_{i} /{\cal F})$ define
$T(\bar{a},X)$ as follows: \\
\\
\centerline{$T(\bar{a},X) = \{ i \in I: (\exists \bar{b} \in
X)(\hat{\bar{a}}(i) = \hat{\bar{b}}(i)) \}$.}
\end{definition}

A set $X \subseteq {}^{k}(\Pi_{i \in I} {\cal A}_{i} /{\cal F})$
is defined to be closed with respect to $\hat{\ }$, iff for all
$\bar{a} \in {}^{k}(\Pi_{i
\in I} {\cal A}_{i} /{\cal F})$ we have \\
\\
\centerline{$T(\bar{a},X) \in {\cal F}$ implies $\bar{a} \in X$.}
\\
\\
\indent Let $\hat{\ }$ be a fixed $k$-dimensional choice
function. For completeness, we note, that the family of $k$-ary
relations $X \subseteq {}^{k}(\Pi_{i \in I} {\cal A}_{i} /{\cal
F})$ closed with respect to $\hat{\ }$, is the family of all
closed sets of a topological space on ${}^{k}(\Pi_{i \in I} {\cal
A}_{i} /{\cal F})$. These spaces are called {\em
ultratopologies}. \\
\indent By Theorem
3.8 of \cite{cl1a} a $k$--ary relation $X$ on an ultraproduct
$A=\Pi_{i \in I} {\cal A}_{i} /{\cal F}$ is decomposable if and
only if there is a $k$--dimensional ultratopology on $A$ under
which $X$ is clopen, see also Theorem 2.1
and Corollary 2.2
in \cite{ut}. For further details we refer to \cite{cl1a} and
\cite{ut}. \\
\\
{\bf Lower Cofinalities of Ultrafilters.} Suppose ${\cal F}
\subseteq {\cal P}(I)$ is an ultrafilter and let $\kappa$ be a
regular cardinal. The natural order of $\kappa$ is the ordering
relation of ordinals restricted to $\kappa$; it will be denoted
by $<_{\kappa}$. An element $a \in {}^{I}\kappa/{\cal F}$ is
defined to be unbounded iff for all $n \in \kappa$ we have $\{i
\in I: \hat{a}(i) \geq n \} \in {\cal F}$. Recall Definition
VI.3.5 from \cite{clasth}: the lower cofinality $lcf(\kappa,{\cal
F})$ is the smallest cardinality $\lambda$ such that there exists
a $\lambda$-sized set of unbounded elements of ${}^{I}\kappa/{\cal
F}$ which is unbounded from below (according to the
ultraproduct-order ${}^{I}<_{\kappa}/{\cal F}$ of the natural
order $<_{\kappa}$ of $\kappa$).

\section{Proofs}
\label{proofs}
\ \indent In this section we present the lemmas and
their proofs we need to
establish our main result: Theorem \ref{mainthm}. \\
\indent Throughout $I$ is a set, $A_i=\{0,1,...,|A_{i}|-1\}$ is a
finite set (that is, $A_{i}$ is an initial segment of the set of
natural numbers), ${\cal G}_{i} = \langle A_{i},E_{i} \rangle$ is
a simple, undirected graph for all $i \in I$ and ${\cal F}
\subseteq {\cal P}(I)$ is a nonprincipal ultrafilter on $I$. In
addition, $A = \prod_{i\in I} A_{i}/{\cal F}$ and ${\cal G} =
\prod_{i \in I} {\cal G}_{i}/{\cal F}$. Throughout the paper we
tacitly assume, that the ultraproduct $A$ is infinite. To denote
the vertex set of ${\cal G}_{i}$, in place of $A_{i}$, we may
write $V({\cal G}_{i})$, as well.

\begin{lemma}
\label{ladderlemma}
There exists a sequence $\langle f_{\alpha} \in {}^{I}\omega:
\alpha < lcf(\aleph_{0},{\cal F}) \rangle$
such that the following propositions hold for all $\alpha,\beta < lcf(\aleph_{0},{\cal F})$ and $n \in \omega$. \\
\indent (i) $f_{\alpha}$ is unbounded, that is, $\{ i \in I:
f_{\alpha} (i) \geq n \} \in {\cal F}$; \\
\indent (ii) $\alpha < \beta$ implies $f_{\beta} <_{\cal F}
f_{\alpha}$, that is, $\{i \in I: f_{\beta}(i) < f_{\alpha}(i) \}
\in {\cal F}$; \\
\indent (iii) for all unbounded $h \in {}^{I}\omega$ there exists
$ \alpha < lcf(\aleph_{0},{\cal F})$ with $f_{\alpha} <_{\cal F}
h$.
\end{lemma}

\begin{proof}
Let $\{g_{\alpha}: \alpha < lcf(\aleph_{0},{\cal F}) \}$ be a
cofinal subset of the set of unbounded elements of
${}^{I}\aleph_{0}/{\cal F}$ (ordered by the reverse of the
ultrapower of the natural ordering of $\omega$). We define the
sequence $\langle f_{\alpha} \in {}^{I}\omega: \alpha <
lcf(\aleph_{0},{\cal F}) \rangle$ by transfinite recursion as
follows. Suppose $\gamma < lcf(\aleph_{0},{\cal F})$ and for any
$\alpha < \gamma$ the function $f_{\alpha}$ has already been
defined such that the following stipulations are satisfied: \\
\indent (a) $f_{\alpha}$ is unbounded; \\
\indent (b) $\alpha < \beta < \gamma $ implies  $f_{\beta} <_{\cal
F} f_{\alpha}$; \\
\indent (c) $\alpha < \beta < \gamma $ implies $f_{\beta} <_{\cal
F} g_{\alpha}$. \\
Let $F=\{f_{\alpha}, g_{\alpha}: \alpha < \gamma \}$. Since
$\gamma < lcf(\aleph_{0},{\cal F})$, it follows, that $|F| <
lcf(\aleph_{0},{\cal F})$, hence, there exists an unbounded
$f_{\gamma} \in {}^{I} \omega$ which is a lower bound of the set
$F$. Clearly, (a)-(c) remain true. In this way the sequence
$\langle f_{\alpha} \in {}^{I}\omega: \alpha <
lcf(\aleph_{0},{\cal F}) \rangle$ can be completely built up.
Clearly, (a) implies (i) and (b) implies (ii). Since
$\{g_{\alpha}: \alpha < lcf(\aleph_{0},{\cal F}) \}$ is cofinal,
(c) implies (iii).
\end{proof}

\begin{definition}
Let $B$ be a set. A function $X: {}^{< \omega}2 \rightarrow {\cal
P}(B)$ is defined to be a Hausdorff scheme over $B$ iff the
following stipulations are satisfied for all $p\in {}^{<
\omega}2$: \\
\indent $X_{p^{\frown}0} \cap X_{p^{\frown}1} = \emptyset$; \\
\indent $X_{p^{\frown}0}, X_{p^{\frown}1} \subseteq X_{p}$. \\
Thus, a Hausdorff scheme is an infinite binary tree whose
vertices are colored with certain subsets of $B$. Above, and
thereafter we write $X_{p}$ instead of
$X(p)$. \\
\indent A Hausdorff scheme over $A$ is defined to be decomposable
iff $X_{p}$ is a decomposable subset of $A$ for any $p \in {}^{<
\omega}2$.
\end{definition}

\begin{definition}
Let $Y = \langle Y_{i}, i \in I \rangle /{\cal F}$ be a decomposed
subset of $A$. If $|Y_{i}| \geq 1$ and $|A_{i}|\geq 2$, then the
size $\mu_{i}(Y)$ of \ $Y$ at $i$ is defined to be
$$\mu_{i}(Y) = \frac{log(|Y_{i}|)}{log(|A_{i}|)}.$$
\end{definition}

\begin{definition}
Let $X = \langle X_{i}: i \in I \rangle /{\cal F}$ be a
decomposable subsets of $A$.  $X$ is defined to be big iff there exists $c \in \Re^{+}$ such that \\
\\
\centerline{$(*) \indent \{i \in I: |X_{i}| \geq |A_{i}|^{c} \} \in {\cal F}$.} \\
\\
We note, that although a decomposable $X \subseteq A$ may have
several different decompositions, the truth of $(*)$ does not
depend on the particular choice of the decomposition $\langle
X_{i}: i \in I \rangle /{\cal F}$.
\\
\indent In addition, $X$ is defined to be small iff it is
(decomposable and) not big.
\end{definition}

\begin{definition}
Let $X = \langle X_{p}: p \in {}^{<\omega}2 \rangle$ be a
decomposed Hausdorff scheme of $A$. Then the depth $\delta_{i}(X)$
of $X$ at $i$ is defined to be
$$ \delta_{i}(X) = max \{k \in \omega: \ (\forall p \in {}^{k}2)((X_{p})_{i} \not=
\emptyset) \}). $$
\end{definition}

\begin{definition}
Suppose $\langle f_{\alpha} \in {}^{I}\omega: \alpha <
lcf(\aleph_{0},{\cal F}) \rangle$ is a sequence of unbounded
functions satisfying the consequences of Lemma \ref{ladderlemma}.
Let $Y = \langle Y_{i}, i \in I \rangle /{\cal F}$ and $Z =
\langle Z_{i}, i \in I \rangle /{\cal F}$ be decomposable subsets
of $A$ and let $\alpha < lcf(\aleph_{0},{\cal F})$. Then $Z$ is
defined to be $\alpha$-close to $Y$ iff
$$\{i \in I: \mu_{i}(Z) \geq \mu_{i}(Y)-
\frac{f_{\alpha}(i)}{log(|A_{i}|)} \} \in {\cal F}.$$
\end{definition}

\begin{remark} {\em Keeping the notation introduced in the above
definitions, we make the following remarks. \\
\indent (1) Note, that for all $i \in I$ we have
$|Y_{i}|=|A_{i}|^{\mu_{i}(|Y_{i}|)}$. Consequently, $Y$ is big
iff there exists $c \in \Re^{+}$ such that $$\{i \in I:
\mu_{i}(X) \geq c \} \in {\cal F}.$$
%
\indent (2) Combining the facts, that $A_{i}$ is finite and
$(X_{p})_{i} \supseteq (X_{p^{\frown}t})_{i}$ for any $t \in 2$,
one concludes that the set $$\{k \in \omega: \ (\forall p \in
{}^{k}2)((X_{p})_{i} \not= \emptyset) \}$$ is finite, hence it
has a maximum. \\
\indent (3) Note, that $\alpha$-closeness is not a symmetric
notion. }
\end{remark}

Let $X = \langle X_{p}: p \in {}^{<\omega}2 \rangle$ be a
decomposable Hausdorff scheme of $A$. For $i \in I$ we say, that
$T_{i} \subseteq A_{i}$ is a transversal set at $i$ iff there
exists $l \in \omega$ such that for all $p \in {}^{l}2$ we have
$|T_{i} \cap (X_{p})_{i}| = 1$. In addition, we say, that $T
\subseteq A$ is a transversal set of $X$ iff $T = \langle T_{i}:
i \in I \rangle/{\cal F}$ is decomposable and $\{i \in I: T_{i}$
is a transversal set at $i \} \in {\cal F}$.

\begin{definition}
\label{septree} Let $n \in \omega$. A function $Z: {}^{\leq n}2
\rightarrow {\cal P}(A_{i})$ is defined to be a separating tree
of height $n$ in ${\cal G}_{i}$ iff
for all $p \not=q \in {}^{<n}2$ and $i \in 2$ we have \\[2mm]
\indent (a) $Z_{p^{\frown}i} \subseteq Z_{p} \not=\emptyset$; \\
\indent (b) $Z_{p^{\frown}0} \cap Z_{p^{\frown}1} = \emptyset$; \\
\indent (c) either $$\Big(\forall a \in
Z_{p^{\frown}0}\Big)\Big(\forall b \in Z_{p^{\frown}1}\Big) \
\langle a,b \rangle \in E({\cal G}_{i}) $$
\indent \indent or
$$\Big(\forall a \in Z_{p^{\frown}0}\Big)\Big(\forall b \in
Z_{p^{\frown}1}\Big) \ \langle a,b \rangle \not\in E({\cal
G}_{i}).$$
\\
Above, and sometimes thereafter we write $Z_{p}$ instead of
$Z(p)$.
\end{definition}

\begin{lemma}
\label{bigtransversal} Let $X = \langle X_{p}: p \in
{}^{<\omega}2 \rangle$ be a decomposed Hausdorff scheme of $A$.
Suppose, there exists $c \in \Re^{+}$ such that
$$ I_{0}:=\{i \in I: \delta_{i}(X) \geq c \cdot log(|A_{i}|) \} \in {\cal F}.$$ Then
there exists a big $T\subseteq A$ which is a transversal set of
$X$.
\end{lemma}

\begin{proof}
Clearly, for each $i \in I_{0}$ there exists $T_{i} \subseteq
A_{i}$ such that for all $p \in {}^{\delta_{i}(X)}2$ we have
$|T_{i} \cap (X_{p})_{i} | = 1$. Then, by definition, $T_{i}$ is
a transversal set at $i$, hence $T:=\langle T_{i}: i \in I
\rangle /{\cal F}$ is a transversal set of $X$. It remains to
show, that $T$ is big. \\
\indent To do so, let $i \in I_{0}$ be arbitrary, and observe,
that $$|T_{i}| = 2^{\delta_{i}(X)} \geq 2^{ c \cdot log(|A_{i}|)}
= |A_{i}|^{c}.$$ Consequently, $T$ is big, as desired.
\end{proof}

\begin{lemma}
\label{edgelemma}
Suppose $\langle f_{\alpha} \in {}^{I}\omega: \alpha <
lcf(\aleph_{0},{\cal F}) \rangle$ is a sequence of unbounded
functions satisfying the consequences of Lemma \ref{ladderlemma}.
Suppose, that if $X \subseteq V({\cal G})$ is decomposable and
${\cal G}|_{X}$ is either a complete or empty graph, then $X$ is
small. \\
\indent Then there exists  $\alpha < \lambda:= lcf(\aleph_{0},
{\cal F})$ with the following property. If $X \subseteq V({\cal
G})$ is big, then there exists
$X' \subseteq X$ such that \\
\indent (i) $X'$ is big; \\
\indent (ii) if $V,W \subseteq X'$ are disjoint and
$\alpha$-close to $X'$, then there are $a,a' \in V$ and $b,b' \in
W$ such that $\langle a,b \rangle \in E({\cal G})$ and $\langle
a',b' \rangle
\not\in E({\cal G})$. \\[1mm]
\indent In addition, the function $\varepsilon: I \rightarrow
\Re,  \ \displaystyle{ \varepsilon(i) := \frac{log(|A_{i}|)}{
f_{\alpha}(i)} }$ is unbounded.
\end{lemma}

\begin{proof}
For each $i \in I$ let us fix a separating tree $Z_{i}$ at ${\cal
G} _{i}$ with largest possible height $g(i)$ (such a separating
tree with largest height exists because $A_{i}$ is finite). For
each $p \in {}^{< \omega}2$ and $i \in I$ let
\[(X_{p})_{i}  =  \left\{ \begin{array}{ll}
                                 Z_{i}(p) & \mbox{if  $Z_{i}(p)$ is defined, } \\
                                 \emptyset  &   \mbox{otherwise.}\\
                            \end{array}
                    \right. \]
Then $X:=\langle (X_{p})_{i}: i \in I \rangle /{\cal F}: p \in
{}^{< \omega}2 \rangle$ is a decomposed Hausdorff scheme such
that for all $i \in I$ we have $\delta_{i}(X) = g(i)$. \\
\indent We claim, that for all $c \in \Re^{+}$ we have $$(*)
\indent \{i \in I: g(i) \leq c\cdot log(|A_{i}|) \} \in {\cal
F}.$$ To see this, first we define a coloration $t_{i}: {}^{<
g(i)}2 \rightarrow 2$ of the non-terminal nodes of $Z_{i}$ as
follows. For any $i \in I$ and $p \in {}^{< g(i)}2$ let $t_{i}(p)
= 0 $ iff $Z_{i}$ satisfies the first case of Definition
\ref{septree} (c) and let $t_{i}(p) = 1$ otherwise. Then, by (the
proof of) Lemma 6.7.9 of \cite{Hodges} there exists a complete
subtree $Z_{i}'$ of $T_{i}$ with height at least $g(i)/2-1$ which
is monochromatic under $t_{i}$. Let $t'_{i}$ be the color for
which $T_{i}'$ is monochromatic. Then there exists $t \in 2$ and
$J \in
{\cal F}$ such that for all $i \in J$ we have $t_{i}'=t$. \\
\indent Now, assume, seeking a contradiction, that $(*)$ does not
hold, that is, there exists $c \in \Re^{+}$ such that $$L:=\{i
\in I: \delta_{i}(X) \geq c \cdot log(|A_{i}|) \} \in {\cal F}.$$
For each $p \in {}^{< \omega}2$ and $i \in I$ let
\[(X_{p}')_{i}  =  \left\{ \begin{array}{ll}
                                 Z_{i}'(p) & \mbox{if  $Z_{i}'(p)$ is defined, } \\
                                 \emptyset  &   \mbox{otherwise.}\\
                            \end{array}
                    \right. \]
Then $X':=\langle (X_{p}')_{i}: i \in I \rangle /{\cal F}: p \in
{}^{< \omega}2 \rangle$ is a decomposed Hausdorff schema such
that for all $i \in L \cap J \cap \{i \in I: c \cdot log(|A_{i}|)
\geq 6 \} \in {\cal F}$ we have
$$\delta_{i}(X') \geq \frac{\delta_{i}(X)}{2}-1 \geq \frac{c \cdot
log(|A_{i}|)}{2}-1 \geq \frac{c \cdot log(|A_{i}|)}{3}.$$ Hence,
by Lemma \ref{bigtransversal} there exists a big set $T= \langle
T_{i}: i \in I \rangle/{\cal F}$ transversal to $X'$. It follows,
that for any $i \in
J$, if $t=0$ then $T_{i}$ spans a complete subgraph of ${\cal
G}_{i}$ and if $t=1$ then $T_{i}$ spans an empty subgraph of
${\cal G}_{i}$. Since $T$ is big, both cases contradict to the
assumptions of the Lemma. Hence
$(*)$ holds, as desired. \\
\indent Because of $(*)$, the function $i \mapsto
\frac{log(|A_{i}|)}{g(i)}$ is unbounded. Therefore, there exists
$\alpha < \lambda$ such that $$I_{0}:=\{i \in I: f_{\alpha}(i)
\leq \sqrt{\frac{log(|A_{i}|)}{g(i)}} \} \in {\cal F}.$$ We
claim, that this $\alpha$ satisfies the conclusion of the lemma.
First observe, that the function $\varepsilon(i)=
\frac{log(|A_{i}|)}{f_{\alpha}(i)}$ is unbounded, because for all
$i \in I_{0}$ we have
$$\frac{log(|A_{i})|)}{f_{\alpha}(i)} \geq
\sqrt{g(i) \cdot log(|A_{i}|)}$$ and the latter
function is clearly unbounded (because the second factor of the right hand side is unbounded by assumption). \\
\indent To check the remaining part of the statement of the
present lemma, let $X \subseteq A$ be big. A separating tree
$Z_{i}:{}^{\leq n}2 \rightarrow {\cal P}(A_{i})$ at ${\cal G}_{i}$
will be called $\alpha$-large iff $Z_{i}(\langle \rangle)=X_{i}$
and for all $p \in {}^{<n}2$ and $r \in 2$ we have
$\mu_{i}(Z_{i}({p^{\frown}r})) \geq \mu_{i}(Z_{i}(p)) -
\frac{f_{\alpha}(i)}{log(|A_{i}|)}$. For each $i \in I$ choose an
$\alpha$-large separating tree $Z_{i}$ at ${\cal G}_{i}$ with
largest possible height $h(i)$. Since the height of $Z_{i}$ is
largest possible, it follows, that for any $i \in I$ there exists
$p_{i}
\in {}^{h(i)}2$ such that \\
\\
\indent $(**)$ \ \ \  if $V,W \subseteq Z_{i}(p_{i})$ are disjoint
and $$\mu_{i}(V) \geq \mu_{i}(Z_{i}(p)) -
\frac{f_{\alpha}(i)}{log(|A_{i}|)}, \ \ \mu_{i}(W) \geq
\mu_{i}(Z_{i}(p)) - \frac{f_{\alpha}(i)}{log(|A_{i}|)}$$ then
there are $a,a' \in A$ and $b,b' \in B$ such that $\langle a,b
\rangle \in E({\cal G}_{i})$ and $\langle a',b' \rangle \not\in
E({\cal G}_{i})$. \\
\\
Let $X_{i}'=Z_{i}(p_{i})$ and let $X' = \langle X_{i}': i \in I
\rangle /{\cal F}$. We claim, that this $X'$ satisfies the
conclusion of the present Lemma. Combining $(**)$ and the {\L}o\'s
Lemma, it follows, that $X'$ satisfies (ii) of the present lemma.
It remains to show (i). $X'$ is big because $X$ is big and hence
$$I_{1}:= \{i \in I_{0}: \ \sqrt{\frac{g(i)}{log(|A_{i}|)}} \leq
\frac{\mu_{i}(X)}{2} \} \in {\cal F}.$$ Moreover, for all $i \in
I_{1} \subseteq I_{0}$ we have
$$ \mu_{i}(X') \geq \mu_{i}(X) - \frac{h(i)
f_{\alpha}(i)}{log(|A_{i}|)} \geq
\mu_{i}(X)-\frac{g(i)f_{\alpha}(i)}{log(|A_{i}|)} \ \stackrel{i
\in I_{0}}{\geq}
$$
$$ \mu_{i}(X)-\frac{g(i)}{log(|A_{i}|)} \cdot
\sqrt{\frac{log(|A_{i}|)}{g(i)}} = \mu_{i}(X) -
\sqrt{\frac{g(i)}{log(|A_{i})}} \ \stackrel{i \in I_{1}}{\geq}
\frac{\mu_{i}(X)}{2}.$$

\end{proof}

The following definition is motivated by Keisler's limit
ultrapower and limit ultraproduct constructions, see
\cite{limitultrapow} and \cite{limitultraprod}. Our version below
is slightly different, so instead of recalling Keisler's original
construction, below we are examining our variant which seems to
be more adequate for the purposes of the present paper.

\begin{definition}
Suppose $n \in \omega$ and $A \supseteq Y_{k} = \langle
(Y_{k})_{i}: i \in I \rangle /{\cal F}$ are decomposed relations
for all $ k < n$. Then the canonical equivalence relation
$$\Theta = \Theta(\langle \langle {\cal G}_{i},
(Y_{k})_{i}\rangle_{k < n}: i \in I \rangle)$$ is defined to be $$
\{ \langle i,j \rangle \in I \times I: {\cal G}_{i} = {\cal
G}_{j}, \ (\forall k < n)((Y_{k})_{i}=(Y_{k})_{j}) \}.$$ In
addition, if \ $\hat{\ }$ is a choice function on $A$ and $\Theta$
is an equivalence relation on $I$ then the substructure ${\cal
A}_{|\Theta}$ is defined to be

$$ {\cal A}_{|\Theta} = \{s \in A: (\exists J \in {\cal
F})(\forall i, j \in J)(\langle i,j \rangle \in \Theta
\Rightarrow \hat{s}(i) = \hat{s}(j)) \}.$$ Note, that ${\cal
A}_{|\Theta}$ does not depend on the particular choice of \
$\hat{\ }$.
\end{definition}

\begin{lemma}
\label{elemlemma} Suppose $U = \langle U_{i}: i \in I \rangle
/{\cal F}$ and $V = \langle V_{i}: i \in I \rangle /{\cal F}$ are
decomposed subsets of $A$. Suppose $\Theta$ is an equivalence
relation on $I$ contained by the the canonical equivalence
relation of $\langle \langle {\cal A}_{i}, U_{i}, V_{i}\rangle: \
i \in I \rangle$.
\\[1mm]
\indent (1) If there exist $a \in U$ and $b \in V$ such that
$\langle a,b \rangle \in E({\cal A})$ then there exists $a' \in U
\cap A_{|\Theta}, b' \in V \cap A_{|\Theta}$ with $\langle
a',b'\rangle \in E({\cal A})$. \\[1mm]
\indent (2) If there exist $a \in U$ and $b \in V$ such that
$\langle a,b \rangle \not\in E({\cal A})$ then there exists $a'
\in U \cap A_{|\Theta}, b' \in V \cap A_{|\Theta}$ with $\langle
a',b'\rangle \not \in E({\cal A})$. \\
\end{lemma}

\begin{proof}
First we show (1). Let $\hat{\ }$ be any choice function on
${\cal A}$. Assume $a \in U$ and $b \in V$ are such that $\langle
a,b \rangle \in E({\cal A})$. Let $$J:= \{i \in I: \langle
\hat{a}_{i},\hat{b}_{i} \rangle \in E({\cal A}_{i}), \ \hat{a}_{i}
\in U_{i}, \ \hat{b}_{i} \in V_{i} \}.$$ Because of the
assumptions, $J \in {\cal F}$. For any $e \in I/{\Theta}$ let
$e^{*} \in e$ be a representative of $e$ such that if $e \cap J
\not=\emptyset$ then $e^{*} \in J$ also holds. For any $i \in I$
let
\[a'_{i}  =  \left\{ \begin{array}{ll}
                                 \hat{a}_{(i/{\Theta})^{*}} & \mbox{if  $(i/{\Theta})^{*} \in J$, } \\
                                 \hat{a}_{i}  &   \mbox{otherwise}\\
                            \end{array}
                    \right. \]
and similarly, let
\[b'_{i}  =  \left\{ \begin{array}{ll}
                                 \hat{b}_{(i/{\Theta})^{*}} & \mbox{if  $(i/{\Theta})^{*} \in J$, } \\
                                 \hat{b}_{i}  &   \mbox{otherwise.}\\
                            \end{array}
                    \right. \]
Finally, let $a' = \langle a_{i}':i \in I \rangle /{\cal F}$ and
let $b' = \langle b_{i}':i \in I \rangle /{\cal F}$.
Then $J$ witnesses $a',b' \in A_{|\Theta}$ and clearly, (by the
assumptions on $\Theta$) we have $a' \in U$, $b' \in V$ and
$\langle a',b' \rangle \in E({\cal A})$.
This completes the proof of (1). \\
\indent (2) can be proved completely similarly; as an alternative
proof, one can apply (1) directly to the complementer graph of
${\cal A}$.
\end{proof}

\begin{lemma}
\label{countlemma} Suppose $\Theta$ is an equivalence relation on
$I$ with countably many equivalence classes which is contained in
the canonical equivalence relation of $\langle {\cal A}_{i}: i
\in I \rangle$. Then $|{\cal A}_{| \Theta}| \leq 2^{\aleph_{0}}$.
\end{lemma}

\begin{proof}
For any $e \in
 I/{\Theta}$ let $e^{*} \in e$ be a fixed representative.
Let
$${\cal B} = \prod_{e \in I/{\Theta}} {\cal A}_{e^{*}}.$$ Since $\Theta$ has countably many equivalence classes,
$|B| \leq 2^{\aleph_{0}}$. For any $s \in B$ and $i \in I$ let
$s'_{i} = s_{i /{\Theta}}$ and let $f(s) = \langle s'_{i}: i \in
I \rangle /{\cal F}$. Then it is easy to check, that $f: {\cal B}
\rightarrow {\cal A}_{|\Theta}$ is well defined and surjective;
this completes the proof.
\end{proof}

Recall, that an ultrafilter ${\cal F}$ is defined to be {\em
$\kappa$-regular} iff it contains a point-finite subset of
cardinality $\kappa$, that is, ${\cal F}$ is $\kappa$-regular iff
there exists $E \in [{\cal F}]^{\kappa}$ such that for all $i \in
I$ the set $\nu(i):=\{e \in E: i \in e \}$ is finite. For further
details we refer the reader to Section 4.3 of \cite{chk}. \\
\indent Recall also, that for a cardinal $\kappa$ the ultrafilter
${\cal F}$ is defined to be $\kappa$-good, if for all $f:
[\kappa]^{< \aleph_{0}} \rightarrow {\cal F}$ there exists $g:
[\kappa]^{< \aleph_{0}} \rightarrow {\cal F}$ such that for all
$s,z \in [\kappa]^{< \aleph_{0}}$ we have $g(s) \subseteq f(s)$
and $g(s \cup z) = g(s) \cap g(z)$. The function $g$ is called an
``additive refinement'' of $f$. We refer to \cite{chk} or
\cite{clasth} as standard references for good ultrafilters. \\
\indent Lemma \ref{flexilemma} below is a special case of Lemmas
8.7 and 8.8 of \cite{Malliaris3} which uses a quite different
terminology. Therefore we include here (a rather standard) proof.
For closely related results we refer to the more recent
\cite{Malliaris4}, as well.

\begin{lemma}
\label{flexilemma} Suppose ${\cal F} \subseteq {\cal P}(I)$ is
$\aleph_{1}$-incomplete and $\kappa^{+}$-good. Then ${\cal F}$ is
$\kappa$-flexible, that is, if $f \in {}^{I}\omega$ is unbounded
modulo ${\cal F}$ then there exists a point finite $E \in [{\cal
F}]^{\kappa}$ such that, for any $i \in I$ we have
$$|\{e \in E: i \in e \}| \leq f(i).$$
\end{lemma}

\begin{proof}
Let $f \in {}^{I}\omega$ be an unbounded function (modulo ${\cal
F}$). Since ${\cal F}$ is countably incomplete, there exists a
decreasing sequence $\langle I_{n}: n \in \omega \rangle$ such
that $I_{n} \in {\cal F}$ for all $n \in \omega$ and $\cap_{n \in
\omega} I_{n} = \emptyset$. For any $w \in
[\kappa]^{<\aleph_{0}}$ let
$$h(w) = I_{|w|} \cap \{i \in I: f(i) \geq |w| \} .$$ Since
${\cal F}$ is $\kappa^{+}$-good, it follows, that there exists an
additive refinement $g: [\kappa]^{< \omega} \rightarrow {\cal F}$
of $h$. In addition, for any $i \in I$ let $\nu(i) = \{\alpha \in
\kappa: i \in g(\{\alpha\}) \}$. Let $i \in I$ be fixed, and
assume $\alpha_{0},...,\alpha_{n-1} \in \nu(i)$. Then
$$(*) \indent i \in g(\{\alpha_{0}\}) \cap ... \cap g(\{\alpha_{n-1}\}) =
g(\{\alpha_{0},...,\alpha_{n-1}\}) \subseteq
h(\{\alpha_{0},...,\alpha_{n-1}\}).$$ Combining this with the
definition of $h$, we obtain $f(i) \geq n$. For any $\alpha <
\kappa$ let $e_{\alpha} = g(\{ \alpha \})$. Then $(*)$ shows,
that $\alpha \mapsto e_{\alpha}$ is a finite-to-one mapping,
hence $E := \{ e_{\alpha}: \alpha < \kappa \} \in [{\cal
F}]^{\kappa}$. It is also easy to see, that by $(*)$, $E$ is
point finite, moreover, satisfies the other part of the statement.
\end{proof}

\begin{definition} Suppose $X_{0},...,X_{n-1}
\subseteq A$ are decomposable.
Then the point $a =\langle a_{i}: i \in I \rangle / {\cal F}\in
A$ is defined to be $ \langle X_{0},...,X_{n-1} \rangle$-generic
iff there exists $c \in \Re^{+}$ such that for all $j < n $ we
have
$$\{ i \in I: |(X_{j})_{i} \cap
\Gamma(a_{i})|, \ \  |(X_{j})_{i} - \Gamma(a_{i})| \ \ \geq c
\cdot |(X_{j})_{i}| \} \in {\cal F}.$$
\end{definition}

\begin{lemma}
\label{genericlemma} Let $\lambda := lcf(\aleph_{0},{\cal F})$.
Suppose $\langle f_{\alpha} \in {}^{I}\omega: \alpha < \lambda
\rangle$ is a sequence of unbounded functions satisfying the
consequences of Lemma \ref{ladderlemma}. Assume ${\cal F}$ is
$\aleph_{1}$-incomplete
and $(2^{\aleph_{0}})^{+}$-good. \\
\indent Suppose $X'$ and $\alpha$ satisfy (i) and (ii) of Lemma
\ref{edgelemma}. Suppose $Y_{0},...,Y_{n} \subseteq X'$ are
disjoint, decomposable sets such that there exists $c \in \Re^{+}$
with
$$I':=\{i \in I: |(Y_{0})_{i}|,...,|(Y_{n})_{i}| \geq c \cdot
|X'_{i}| \} \in {\cal F}.$$ Then \\
\\
\centerline{$(*)$ \indent for all $\beta < \lambda$ there exists
$a_{\beta} \in Y_{n}$ such that for all $j < n $ we have}\\[2mm]
\centerline{$\displaystyle{ \{ i \in I: \ |(Y_{j})_{i} \cap
\Gamma((\hat{a}_{\beta})_{i})| \ \geq \ \frac{1}{
f_{\beta}(i)} \cdot |(Y_{n})_{i}| \} \in {\cal F}}$ and} \\[2mm]
\centerline{ $\displaystyle{ \{i \in I: \  |(Y_{j})_{i} -
\Gamma((\hat{a}_{\beta})_{i})| \ \geq \ \frac{1}{ f_{\beta}(i)}
\cdot |(Y_{n})_{i}| \} \in {\cal F}}$.}
%
%
\end{lemma}

\begin{proof}
Fix a choice function $\hat{\ }:A \rightarrow \prod_{i \in I}
A_{i}$ and for all $j \leq n$ fix decompositions $Y_{j} = \langle
(Y_{j})_{i}: i \in I \rangle /{\cal F}$.
Assume, seeking a contradiction, that $(*)$ is not true and fix
$\beta < \lambda$ showing this. For all $i \in I$ and $j < n$
define $V(i,j,0), V(i,j,1) \subseteq A_{i}$ to be
$$V(i,j,0):= \{a \in (Y_{n})_{i}: |(Y_{j})_{i} \cap \Gamma (a)| \ <
\ \frac{1}{ f_{\beta}(i)} \cdot |(Y_{n})_{i}| \}$$ and
$$V(i,j,1):=
\{a \in (Y_{n})_{i}: |(Y_{j})_{i} - \Gamma(a)| \ < \ \frac{1}{
f_{\beta}(i)} \cdot |(Y_{n})_{i}| \}.$$ Let $I_{0}:= \{i \in I':
(Y_{n})_{i} = \cup_{j < n,k \in 2} V(i,j,k) \}$. Then $I_{0} \in
{\cal F}$ (because otherwise, for any $i \in I-I_{0} \in {\cal
F}$ there would exist $a_{i} \in (Y_{n})_{i} - \cup_{j < n,k \in
2} V(i,j,k) $ so for all $i
\in I-I_{0}$ and $j < n$ we would have \\
\\
\centerline{$ \displaystyle{|(Y_{j})_{i} \cap \Gamma(a_{i})| \
\geq \ \frac{|(Y_{n})_{i}|}{ f_{\beta}(i)} } $ \ \ \ and \ \ \
$\displaystyle{|(Y_{j})_{i} - \Gamma(a_{i})| \ \geq \
\frac{|(Y_{n})_{i}|}{ f_{\beta}(i)} }$} \\
\\
contradicting to our indirect assumption). By elementary
counting, for any $i \in I_{0}$ there exist $j_{i} < n$ and
$k_{i} \in 2$ with $|V(i,j_{i},k_{i})| \geq
\frac{1}{2n}|(Y_{n})_{i}|$. In addition, there exist $j^{*} < n$
and $k^{*} \in 2$ such that $$I_{1}:=\{i \in I_{0}: j^{*}=j_{i},
k^{*}=k_{i} \} \in {\cal F}.$$ \indent For any $i \in I$ let
$g(i)$ be the largest integer number which is smaller (or equal)
with $\sqrt{f_{\beta}(i)}$. By the assumptions of the present
lemma, Lemma \ref{flexilemma} may be applied: there exists a
family $\{e_{\beta}:  \beta < 2^{\aleph_{0}}\} \subseteq {\cal
F}$ such that for all $i \in I$ we have
$$ |\{ \beta < 2^{\aleph_{0}}: i \in e_{\beta} \}| \leq g(i).$$
Let $\Theta$ be the canonical equivalence relation of $\langle
{\cal A}_{i}, \ V(i,j^{*},k^{*}), \ (Y_{j^{*}})_{i} \rangle)_{i
\in I}$. By Lemma \ref{countlemma} we have $|{\cal A}_{|\Theta} |
\leq 2^{\aleph_{0}}$. Fix an enumeration $A_{|\Theta} =
\{a_{\gamma}: \gamma < 2^{\aleph_{0}} \}$. \\
For all $i \in I_{1}$ define $W_{i} \subseteq A_{i}$ as follows
\[W_{i} :=  \left\{ \begin{array}{ll}
                                 (Y_{j^{*}})_{i} - \cup \{ \Gamma((\hat{a}_{\gamma})_{i}): \gamma < 2^{\aleph_{0}}, \ i \in e_{\gamma}, \ (\hat{a}_{\gamma})_{i} \in V(i,j^{*},k^{*}) \} & \mbox{if  $k^{*}=0 $, }
                                 \\[2mm]
                                 (Y_{j^{*}})_{i}  - \cup \{ (Y_{j^{*}})_{i} - \Gamma((\hat{a}_{\gamma})_{i}): \gamma < 2^{\aleph_{0}}, \ i \in e_{\gamma}, \ (\hat{a}_{\gamma})_{i} \in V(i,j^{*},k^{*}) \} &   \mbox{if $k^{*}=1$.}\\
                            \end{array}
                    \right. \]
Next, we show, that \\
\\
\centerline{for any $i \in I_{2}:= I_{1} \cap \{i \in I:
\frac{|(Y_{n})_{i}|}{\sqrt{f_{\beta}(i)}} \leq
\frac{1}{2}|(Y_{j^{*}})_{i}| \}$ we have $|W_{i}| \geq
\frac{1}{2}|(Y_{j^{*}})_{i}|$.} \\
\\
To check this, fix $i \in I_{2}$.
We proceed by a case distinction. \\
{\bf Case 1:} $k^{*}=0$. Then, on one hand, for any $b \in
V(i,j^{*},0)$ we have
$$|(Y_{j^{*}})_{i} \cap \Gamma(b)| \leq \frac{1}{f_{\beta}(i)}
|(Y_{n})_{i}|$$ and on the other hand $$|\{\gamma <
2^{\aleph_{0}}: i \in e_{\gamma}\}| \leq \sqrt{f_{\beta}(i)}.$$
Combining the last two estimations, we obtain
$$|\cup \{ (Y_{j^{*}})_{i} \cap \Gamma((\hat{a}_{\gamma})_{i}): \gamma < 2^{\aleph_{0}}, \  i \in
e_{\gamma}, \ (\hat{a}_{\gamma})_{i} \in V(i,j^{*},0) \}| \leq $$
$$ \frac{\sqrt{f_{\beta}(i)}}{f_{\beta}(i)}|(Y_{n})_{i}| \ \
\stackrel{i \in I_{2}}{\leq} \ \ \frac{1}{2}|(Y_{j^{*}})_{i}|.$$
Therefore $$|W_{i}| = |(Y_{j^{*}})_{i} - \cup \{
\Gamma((\hat{a}_{\gamma})_{i}): \gamma < 2^{\aleph_{0}}, \  i \in
e_{\gamma}, \ (\hat{a}_{\gamma})_{i} \in V(i,j^{*},k^{*}) \}| =
$$
$$|(Y_{j^{*}})_{i} - \cup \{ (Y_{j^{*}})_{i} \cap
\Gamma(\hat{a}_{\gamma})_{i}: \gamma < 2^{\aleph_{0}}, \ i \in
e_{\gamma}, \ (\hat{a}_{\gamma})_{i} \in V(i,j^{*},k^{*}) \}| \geq
$$
$$\frac{1}{2}|(Y_{j^{*}})_{i}|, $$ as
desired. \\
{\bf Case 2:} $k^{*}=1$. This case can be treated analogously. On
one hand, for any $b \in V(i,j^{*},1)$ we have
$$|(Y_{j^{*}})_{i} - \Gamma(b)| \leq \frac{1}{f_{\beta}(i)}
|(Y_{n})_{i}|$$ and on the other hand$$|\{\gamma <
2^{\aleph_{0}}: i \in e_{\gamma}\}| \leq \sqrt{f_{\beta}(i)}.$$
Combining the last two estimations, we obtain
$$|\cup \{ (Y_{j^{*}})_{i} - \Gamma((\hat{a}_{\gamma})_{i}): \gamma < 2^{\aleph_{0}}, i \in
e_{\gamma}, (\hat{a}_{\gamma})_{i} \in V(i,j^{*},1) \}| \leq $$ $$
\frac{\sqrt{f_{\beta}(i)}}{f_{\beta}(i)}|(Y_{n})_{i}| \ \
\stackrel{i \in I_{2}}{\leq} \ \ \frac{1}{2}|(Y_{j^{*}})_{i}|.$$
Therefore $$|W_{i}| = |(Y_{j^{*}})_{i} - \cup \{ (Y_{j^{*}})_{i} -
\Gamma((\hat{a}_{\gamma})_{i}): \gamma < 2^{\aleph_{0}}, i \in
e_{\gamma}, (\hat{a}_{\gamma})_{i} \in V(i,j^{*},k^{*}) \}| \geq
$$
$$\frac{1}{2}|(Y_{j^{*}})_{i}|,
$$ as
desired. \\
Summing up, $V:= \langle V(i,j^{*},k^{*}): i \in I_{2} \rangle
/{\cal F}$ and $W:= \langle W_{i}: i \in I_{2} \rangle /{\cal F}$
satisfy the following: \\
\indent $\bullet$ $V \cap W = \emptyset$ (because $V \subseteq
Y_{n}$ and $W \subseteq Y_{j^{*}}$); \\
\indent $\bullet$ there exist $c_{v},c_{w} \in \Re^{+}$ such that
\\
\\
\centerline{$I_{V}:=\{i \in I: |V(i,j^{*},k^{*})| \geq c_{v} \cdot
|(X)_{i}'| \} \in {\cal F}$ and} \\[1mm] \centerline{$I_{W}:=\{i \in I:
|W_{i}| \geq c_{w} \cdot |X_{i}'| \} \in
{\cal F}$.} \\
\\
This implies, that $V$ and $W$ are $\alpha$-close to $X'$ because
of the following. By Lemma \ref{edgelemma}, the function
$\varepsilon: I \rightarrow \Re,  \ \displaystyle{ \varepsilon(i)
:= \frac{log(|A_{i}|)}{ f_{\alpha}(i)} }$ is unbounded. Hence,
for any $i \in I_{V}$ we have $$ \mu_{i}(V(i,j^{*},k^{*}))
 \ = \ \frac{log (|V(i,j^{*},k^{*})|)}{log(|A_{i}|)} \ \geq \ \frac{log(c_{v} \cdot
 |(X)_{i}'|)}{log(|A_{i}|)}
\ \geq $$ $$ \mu_{i}(|X_{i}'|) + \frac{log(c_{v}) }{log(|A_{i}|)}
\ \geq \
 \mu_{i}(|X_{i}'|) - \frac{f_{\alpha}(i)}{log(|A_{i}|)}$$
and similarly for $W$. \\
\indent $\bullet$ In addition, if $k^{*}=0$ then for any $a \in
V, \ \ b \in W$ we have $\langle a,b \rangle \not\in E({\cal G})$
because of the following. Suppose, seeking a contradiction, that
$a \in V, b \in W$ and $\langle a,b \rangle \in E({\cal G})$. Let
$\Theta^{*}$ be the canonical equivalence relation of $\langle
{\cal A}_{i}, \ V(i,j^{*},k^{*}), \ W_{i}, \ (Y_{j^{*}})_{i}
\rangle)_{i \in I}$. Clearly, $\Theta^{*} \subseteq \Theta$.
Then, we apply Lemma \ref{elemlemma} (1) to $V,W$ and
$\Theta^{*}$ and obtain
$a' \in V\cap A_{|\Theta}, b' \in W \cap A_{|\Theta}$ such that
$\langle a',b' \rangle \in E({\cal G})$. By construction, there
exists $\gamma < 2^{\aleph_{0}}$ such that $a' = a_{\gamma}$ and
$b' \in \Gamma(a_{\gamma})$, particularly, for all
$$i \in e_{\gamma} \cap \{i \in I: \langle \hat{a}'_{i},\hat{b}'_{i} \rangle \in E({\cal G}_{i}), \ \hat{a}_{i}' \in V(i,j^{*},k^{*}) \} \in {\cal F}$$
we have $\hat{b}'_{i} \not \in W_{i}$, so $b' \not \in W$, which
is a contradiction. Completely similarly, if $k^{*}=1$, then for
any $a \in V, b \in W$ we have $\langle a,b \rangle \in E({\cal
G})$. So either there are no edges between $V$ and $W$ (this
holds if $k^{*}=0$), or there are no non-edges between $V$
and $W$ (this holds if $k^{*}=1$). \\
\indent These stipluations together contradict to the assumption,
that $X'$ satisfies (i) and (ii) of Lemma \ref{edgelemma}. This
contradiction completes the proof.
\end{proof}

\begin{lemma}
\label{generic1lemma} Let $\lambda := lcf(\aleph_{0},{\cal F})$.
Suppose $\langle f_{\alpha} \in {}^{I}\omega: \alpha < \lambda
\rangle$ is a sequence of unbounded functions satisfying the
consequences of Lemma \ref{ladderlemma}. Assume ${\cal F}$ is
$\aleph_{1}$-incomplete, $\lambda$-regular
and $(2^{\aleph_{0}})^{+}$-good. \\
\indent Suppose $X'$ and $\alpha$ satisfy (i) and (ii) of Lemma
\ref{edgelemma}. Suppose $Y_{0},...,Y_{n} \subseteq X'$ are
disjoint, decomposable sets such that there exists $c \in \Re^{+}$
with
$$I_{0}:=\{i \in I: |(Y_{0})_{i}|,...,|(Y_{n})_{i}| \geq c \cdot
|X'_{i}| \} \in {\cal F}.$$ Then there exists $a \in Y_{n}$ which
is $ \langle Y_{0},...,Y_{n-1} \rangle$-generic.
\end{lemma}

\begin{proof}
Fix a choice function $\hat{\ }:A \rightarrow \prod_{i \in I}
A_{i}$. By assumption, ${\cal F}$ is $\lambda$-regular: there
exists a family $\{J_{\beta}': \beta < \lambda \} \subseteq {\cal
F}$ such that for all $i \in I$ the sets $\{\beta < \lambda: i \in
J_{\beta}' \}$ are finite. Shrinking $J_{\beta}'$ if necessary,
we may (and will) assume, that for all $i \in J_{\beta}'$ we have
$f_{\beta}(i) \geq 1$. By $(*)$ of Lemma \ref{genericlemma}, for
all $\beta < \lambda$
there exists $a_{\beta} \in Y_{n}$ such that for all $j < n $ we have \\[2mm]
\centerline{$\displaystyle{ J''_{\beta}:=\{ i \in I: \
|(Y_{j})_{i} \cap \Gamma((\hat{a}_{\beta})_{i})| \ \geq \
\frac{1}{f_{\beta}(i)} \cdot |(Y_{n})_{i}| \} \in {\cal F}}$ and}
\centerline{ $\displaystyle{ J'''_{\beta}:=\{i \in I: \
|(Y_{j})_{i} - \Gamma((\hat{a}_{\beta})_(i))| \ \geq \
\frac{1}{f_{\beta}(i)} \cdot |(Y_{n})_{i}| \} \in {\cal F}}$.} \\
\\
For all $\beta < \lambda$ let $$J_{\beta}:=J'_{\beta} \cap
J''_{\beta} \cap J'''_{\beta} \cap \{i \in I:
(\hat{a}_{\beta})_{i} \in (Y_{n})_{i} \}$$ and for any $i \in I$
let $\nu(i):= \{ \beta < \lambda: i \in J_{\beta} \}$. By
construction, for each $i \in I$ we have
$$|\nu(i)| \leq |\{\beta \in \lambda: i \in J'_{\beta} \}|$$ 
hence each $\nu(i)$ is finite. Note, that by the ``shrinking
step'' of the construction of the $J_{\beta}'$, it follows, that
for all $i \in I$ and $\rho \in \nu(i)$ we have $f_{\rho}(i) \geq
1$. For each $i\in I$ let $\varrho(i) \in \nu(i)$ be such that
$$f_{\varrho(i)}(i) = min \{f_{\beta}(i): \beta \in \nu(i) \} \ \ \ (\geq 1)$$
and let $a_{i} = \hat{a}_{\varrho(i)}(i)$. Clearly, $a:= \langle
a_{i}: i \in I \rangle /{\cal F} \in Y_{n}$. We claim, that \\
\\
\indent $(**)$ \indent for all $\beta < \lambda$ and
$j < n $ we have \\[2mm]
\centerline{$\displaystyle{ \{ i \in I: \ |(Y_{j})_{i} \cap
\Gamma(a_{i})| \ \geq \ \frac{1}{f_{\beta}(i)} \cdot
|(Y_{n})_{i}| \} \in {\cal F}}$ and} \centerline{ $\displaystyle{
\{i \in I: \ |(Y_{j})_{i} - \Gamma(a_{i})| \ \geq \
\frac{1}{f_{\beta}(i)} \cdot |(Y_{n})_{i}| \} \in {\cal F}}$.} \\
\\
To show this, fix $\beta < \lambda$ and $j < n$. Let $i \in
J_{\beta}$ be arbitrary. Then $\beta \in \nu(i)$ and hence
$$|(Y_{j})_{i} \cap \Gamma(a_{i})| = |(Y_{j})_{i} \cap
\Gamma((\hat{a}_{\varrho(i)})_{i})| \geq
 \frac{1}{f_{\varrho(i)}(i)} \cdot
|(Y_{n})_{i}| \geq  \frac{1}{f_{\beta}(i)} \cdot |(Y_{n})_{i}|$$
and similarly,
$$ |(Y_{j})_{i} - \Gamma(a_{i})| =
|(Y_{j})_{i} - \Gamma((\hat{a}_{\varrho(i)})_{i})| \geq
 \frac{1}{f_{\varrho(i)}(i)} \cdot
|(Y_{n})_{i}| \geq  \frac{1}{f_{\beta}(i)} \cdot |(Y_{n})_{i}|.$$
So $(**)$ has been established. \\
\indent For any $i \in I$ let $$\gamma_{i}:= max \{
\frac{|(Y_{n})_{i}|}{|(Y_{j})_{i} \cap \Gamma((\hat{a})_{i})|}, \
\frac{|(Y_{n})_{i}|}{|(Y_{j})_{i} - \Gamma((\hat{a})_{i})|}: \ \
j < n \}.$$ It follows from $(**)$ that for any $\beta < \lambda$
we have $\{i \in I: \gamma_{i} < f_{\beta}(i) \} \in {\cal F}$.
Therefore, by the definition of lower cofinality, there exists $m
\in \omega $ with $\{i \in I:  \gamma_{i} \leq m \} \in {\cal
F}$. Increasing $m$ if necessary, we may (and will) assume $m
\geq 1$ (that is, $m \not=0$). So, for all $j < n$ we have
$$\{i \in I: |(Y_{j})_{i} \cap \Gamma((\hat{a})_{i})|, \ \
|(Y_{j})_{i} - \Gamma((\hat{a})_{i})| \geq \frac{1}{m}
|(Y_{n})_{i}| \} \in {\cal F}$$ as desired.

\end{proof}

\begin{theorem}
\label{maintechnicalthm} Let $\lambda := lcf(\aleph_{0},{\cal
F})$. Suppose $\langle f_{\alpha} \in {}^{I}\omega: \alpha <
\lambda \rangle$ is a sequence of unbounded functions satisfying
the consequences of Lemma \ref{ladderlemma}. Assume ${\cal F}$ is
$\aleph_{1}$-incomplete,
$\lambda$-regular and $(2^{\aleph_{0}})^{+}$-good. \\
\indent Suppose $X'$ and $\alpha$ satisfy (i) and (ii) of Lemma
\ref{edgelemma}. Suppose $Y_{0},...,Y_{n-1} \subseteq X'$ are
disjoint, decomposable sets such that there exists $c \in \Re^{+}$
with
$$\{i \in I: |(Y_{0})_{i}|,...,|(Y_{n-1})_{i}| \geq c \cdot
|X'_{i}| \} \in {\cal F}.$$ \indent If ${\cal H} = \langle
n,E({\cal H}) \rangle$ is a (finite) graph on $n$ vertices then
there exists a function $\varrho: n \rightarrow V({\cal G})$ such
that for any $i < n$ we have $\varrho(i) \in Y_{i}$ and $\varrho$
isomorphically embeds ${\cal H}$ into ${\cal G}$.
\end{theorem}

\begin{proof}
We apply induction on the number of vertices of ${\cal H}$. If
${\cal H}$ has only one vertex, then the statement is trivial.
Now assume, that ${\cal H}$ has $n$ vertices and the theorem is
true for any graph having at most $n-1$ vertices. Let
${\cal H}' = {\cal H}_{|(n-1)}$ be the subgraph of ${\cal H}$
induced by its first $n-1$ vertices. By Lemma \ref{generic1lemma}
there exists $a \in Y_{n-1}$ which is $\langle Y_{0},...,Y_{n-2}
\rangle$-generic. For any $i < n-1$ let
\[Y'_{i}  =  \left\{ \begin{array}{ll}
                                 Y_{i} \cap \Gamma(a) & \mbox{if  $\langle i, n-1 \rangle \in E({\cal H}) $, } \\
                                 Y_{i} - \Gamma(a)  &   \mbox{otherwise.}\\
                            \end{array}
                    \right. \]
Since $a$ is $\langle Y_{0},...,Y_{n-2} \rangle$-generic, there
exists $c' \in \Re^{+}$ such that
$$\{i \in I: |(Y_{0}')_{i}|,...,|(Y_{n-2}')_{i}| \geq c' \cdot
|X'_{i}| \} \in {\cal F}.$$ Applying the induction hypothesis to
${\cal H}'$ and to $\langle Y'_{0},...,Y_{n-2}' \rangle$, we
obtain a function $\varrho': (n-1) \rightarrow A$ such that
$\varrho'(i) \in Y'_{i}$ for all $i < n-1$ and $\varrho'$
isomorphically embeds ${\cal H}'$ into ${\cal G}$. Let $\varrho =
\varrho' \cup \{\langle n-1, a \rangle \}$, that is, let
$\varrho$ be the extension of $\varrho'$ that maps the last
vertex of ${\cal H}$ onto $a$. It is easy to see, that $\varrho$
embeds ${\cal H}$ into ${\cal G}$ such that $\varrho(i) \in
Y_{i}$ holds for all $i < n$.
\end{proof}

\begin{lemma}
\label{ultrafilterlemma} There exist a set $I$ and an ultrafilter
${\cal F} \subseteq {\cal P}(I)$ which is $lcf( \aleph_{0},{\cal
F})$-regular, $\aleph_{1}$-incomplete and
$(2^{\aleph_{0}})^{+}$-good.
\end{lemma}

\begin{proof}
By Theorem VI.3.3.12 of \cite{clasth} there exists a regular
ultrafilter ${\cal F}_{0}$ on $(2^{\aleph_{0}})^{+}$ such that
$lcf(\aleph_{0},{\cal F})=(2^{\aleph_{0}})^{+}$. In addition, by
Theorem VI.3.3.1 of \cite{clasth} there exists a good (in fact,
$(2^{\aleph_{0}})^{+}$-good), $\aleph_{1}$-incomplete ultrafilter
${\cal F}_{1}$ on $2^{\aleph_{0}}$. Let $$I:= (2^{\aleph_{0}})^{+}
\times 2^{\aleph_{0}}, \indent {\cal F}:= {\cal F}_{0} \times
{\cal F}_{1}.$$ Then, by Lemma VI.3.3.7 (1) of \cite{clasth},
${\cal F}$ is $(2^{\aleph_{0}})^{+}$-regular (hence, ${\cal F}$
is also $\aleph_{1}$-incomplete) and  Lemma VI.3.3.7 (2) implies,
that ${\cal F}$ is $(2^{\aleph_{0}})^{+}$-good. To complete the
proof, it is enough to show, that $$(*) \indent
lcf(\aleph_{0},{\cal F}) \leq (2^{\aleph_{0}})^{+}.$$
To do so, assume, that $\langle f_{\alpha} \ \in
{}^{(2^{\aleph_{0}})^{+}}\omega: \alpha < (2^{\aleph_{0}})^{+}
\rangle $ is an unbounded sequence of unbounded functions modulo
${\cal F}_{0}$. By Lemma \ref{ladderlemma} we may assume, that for
all $\beta <  \alpha  < (2^{\aleph_{0}})^{+}$ we have
$$(**) \indent \{i \in (2^{\aleph_{0}})^{+}:  f_{\beta}(i) > f_{\alpha}(i) \} \in {\cal F}_{0}.$$
For each $\alpha < (2^{\aleph_{0}})^{+}$ define the function
$f'_{\alpha}: (2^{\aleph_{0}})^{+} \times 2^{\aleph_{0}}
\rightarrow \omega$ to be $f'_{\alpha}(j,i)=f_{\alpha}(j)$ for all
$i \in 2^{\aleph_{0}}$ and $j \in (2^{\aleph_{0}})^{+}$. Clearly,
each $f'_{\alpha}$ is an unbounded function modulo ${\cal F}$. To
prove $(*)$, it is enough to show, that $\langle f'_{\alpha}:
\alpha < (2^{\aleph_{0}})^{+} \rangle$ is unbounded modulo ${\cal
F}$. Let $a \in {}^{(2^{\aleph_{0}})^{+} \times 2^{\aleph_{0}} }
\omega$ be arbitrary which is unbounded modulo ${\cal F}$. Then,
for any $i < 2^{\aleph_{0}}$ there exists $\alpha_{i} <
(2^{\aleph_{0}})^{+}$ such that $$\{j < (2^{\aleph_{0}})^{+}:
f_{\alpha_{i}}(j) \leq a(j,i) \} \in {\cal F}_{0}.$$ Let $\alpha =
sup \{ \alpha_{i}: i < 2^{\aleph_{0}} \}$. Since
$(2^{\aleph_{0}})^{+}$ is regular, $\alpha <
(2^{\aleph_{0}})^{+}$. Then, by $(**)$ we have $f'_{\alpha+1} <
a$ modulo ${\cal F}$, hence $a/{\cal F}$ is not a lower bound of
$\langle f'_{\alpha}: \alpha < (2^{\aleph_{0}})^{+} \rangle$.
Since $a$ was arbitrary, the sequence $\langle f'_{\alpha}:
\alpha < (2^{\aleph_{0}})^{+} \rangle$ is unbounded modulo ${\cal
F}$, hence $(*)$ holds, as desired.
\end{proof}

\begin{theorem}
\label{mainthm} For each finite graph ${\cal H}$ there exists a
constant $c({\cal H}) \in \Re^{+}$ with $c({\cal H}) \leq 1$ such
that for any finite graph ${\cal G}^{*}$ the following holds: if
${\cal G}^{*}$ has $n$ vertices and does not contain complete and
empty induced subgraphs of size $n^{c({\cal H})}$ then ${\cal H}$
can be isomorphically embedded into ${\cal G}^{*}$.
\end{theorem}

\begin{proof}
Let ${\cal H}$ be a finite graph and assume,
seeking a contradiction, that for any $c \in \Re^{+}, \ c \leq 1$
there exists a finite graph ${\cal G}_{c}$ such that ${\cal
G}_{c}$ does not contain complete and empty induced subgraphs of
size $|V({\cal G}_{c})|^{c}$, but ${\cal H}$ cannot be
isomorphically embedded into ${\cal G}_{c}$. \\
\indent By Lemma \ref{ultrafilterlemma} There exist a set $I$ and
an ultrafilter ${\cal F} \subseteq {\cal P}(I)$ which is $lcf(
\aleph_{0},{\cal F})$-regular, $\aleph_{1}$-incomplete and
$(2^{\aleph_{0}})^{+}$-good. Let $\lambda :=lcf(\aleph_{0},{\cal
F})$. Particularly, there exists $E \in [{\cal F}]^{\lambda}$
such that for any $i \in I$ we have $\nu(i):=\{e \in E: i \in e
\}$ is finite (we will assume, that $I \in E$, thus $\nu(i) \geq
1$, for all $i \in I$). For each $i \in I$ let $c(i) \in \Re^{+}$
be such that $c(i) < 1/\nu(i)
$. Since $1 \leq \nu(i) $ for all $i \in I$, it follows, that
${\cal G}_{c(i)}$ is defined for all $i \in I$. Finally, let
$${\cal G} = \prod_{i \in I} {\cal G}_{c(i)}/{\cal F}.$$
%
%
We claim, that ${\cal G}$ does not contain empty or complete big
induced subgraphs. To show this, let $X = \langle X_{i}: i \in I
\rangle /{\cal F} \subseteq V({\cal G})$ be any decomposable, big
set. Then there exists $c \in \Re^{+}$ such that $I':=\{i \in I:
|X_{i}| \geq |V({\cal G}_{i})|^{c} \} \in {\cal F}$. Let $k \in
\omega$ be such that $1/k < c$, and let $J$ be the intersection
of any $k$ distinct elements of $E$. Then, for any $i \in I' \cap
J$ we have $\nu(i) \geq k$, hence
$$c(i) \ \leq \ \frac{1}{\nu(i)} \  \leq \ \frac{1}{k} \ \leq \
c,$$ so $X_{i}$ does not induce a complete or empty subgraph of
${\cal G}_{c(i)}$. It follows, from the {\L}o\'s Lemma, that $X$
does
not induce a complete or empty subgraph of ${\cal G}$. \\
\indent  Suppose $\langle f_{\alpha} \in {}^{I}\omega: \alpha <
\lambda \rangle$ is a sequence of unbounded functions satisfying
the consequences of Lemma \ref{ladderlemma}. According to the
previous paragraph, the conditions of Lemma \ref{edgelemma} are
satisfied. So, there exists $\alpha < \lambda$ which satisfies the
conclusion of \ref{edgelemma}, in addition, applying the last
five lines of Lemma \ref{edgelemma} to $X=V({\cal G})$, we obtain
$X' \subseteq V({\cal G})$ such
that \\[1mm]
\indent (i) $X'$ is big (particularly, it is decomposable: $X'= \langle X'_{i}: i \in I \rangle/{\cal F}$);
\\[1mm]
\indent (ii) if $V,W \subseteq X'$ are disjoint and
$\alpha$-close to $X'$ then there are $a,a' \in V$ and $b,b' \in
W$ such that $\langle a,b \rangle \in E({\cal G})$ and $\langle
a',b' \rangle
\not\in E({\cal G})$. \\[1mm]
\indent It is easy to see, that $X'$ can be partitioned into
$n:=|V({\cal H})|$-many disjoint decomposable sets
$Y_{0},...,Y_{n-1}$ such that for all $j < n$ we have $$\{i \in
I: |(Y_{j})_{i}| \geq \frac{1}{2n} \cdot |X'_{i}| \} \in {\cal
F}.$$ Then Theorem \ref{maintechnicalthm} can be applied: there
exist $a_{0} \in Y_{0},...,a_{n-1} \in Y_{n-1}$ such that the
subgraph of ${\cal G}$ induced by $\{a_{0},...,a_{n-1} \}$ is
isomorphic to ${\cal H}$. It is well known, that there exists a
first order formula $\delta_{\cal H}$ (called the diagram of
${\cal H}$) such that for any graph ${\cal M}$ we have ${\cal M}
\models \delta_{\cal H}$ iff ${\cal H}$ can be isomorphically
embedded into ${\cal M}$; for more details we refer to \cite{chk}.
Hence - again by the {\L}o\'s Lemma \\
\\
\centerline{ $\{i \in I: {\cal H}$ can be isomorphically
embedded into ${\cal G}_{c(i)} \} \in {\cal F}$,} \\
\\
contradicting to the first sentence of the present proof.
\end{proof}

\section{Concluding Remarks}
\label{conc}

In this section we describe some problems which remained open.

\begin{problem}
In Theorem \ref{maintechnicalthm}, can the conditions on the
ultrafilter ${\cal F}$ be replaced by weaker ones such that
Theorem \ref{maintechnicalthm} remains true ?
\end{problem}

The proof of Theorem \ref{mainthm} above establishes the
existence of $c({\cal H})$, but does not provide methods to
compute, or estimate it from (the structure of) ${\cal H}$.
Hence, the next problem is quite interesting, and remained
completely open.

\begin{problem}
Develop methods estimating $c({\cal H})$ from ${\cal H}$. In
particular, is it true, that $c({\cal H})$ may be chosen to be
$2^{-|V({\cal H})|}$ ?
\end{problem}

In the proof of Theorem \ref{mainthm} we found a single copy of
${\cal H}$ in the ultraproduct graph ${\cal G}$, because for the
present purposes this was enough. However, it is easy to see,
that ${\cal G}$ contains ``many'' isomorphic copies of ${\cal H}$.
In that direction the following problem remained open.

\begin{problem}
Suppose ${\cal H}$ is a finite graph and ${\cal G} = \prod_{i \in
I}{\cal G}_{i} /{\cal F} $ is an ultraproduct of finite graphs
such that if $X \subseteq V({\cal G})$ is decomposable and
induces a complete or empty subgraph of ${\cal G}$ then $X$ is
small. Is it true, that the set of isomorphic copies of ${\cal
H}$ in ${\cal G}$ is big in the following sense: there exists $c
\in \Re^{+}$ such that if $Y = \langle Y_{i}: i\in I \rangle
/{\cal F} \subseteq {}^{|V({\cal H})|}{\cal G}$ is a decomposable
subset containing all the isomorphic copies of ${\cal H}$, then
$$\{ i \in I: |Y_{i}| \geq |V({\cal G}_{i})|^{|V({\cal H})|\cdot c}
\} \in {\cal F} \ ?$$ Is this true, if we assume further
properties of the ultrafilter ${\cal F}$ ?
\end{problem}

In \cite{CFS} it was shown, that the analogue of Theorem
\ref{mainthm} for $k$-uniform hypergraphs is not true, if $k \geq
4$. This motivates our last problem.

\begin{problem}
Can Theorem \ref{mainthm} be generalized to $3$-uniform
hypergraphs ?
\end{problem}

\noindent {\bf Acknowledgements.} I am very grateful to J\'anos
Pach for his valuable comments on an earlier version of this
paper. I am also very grateful to Andr\'as Hajnal and Istv\'an
Juh\'asz for their constant encouragement.

\leftline{Alfr\'ed R\'enyi Institute of Mathematics \indent
\indent \ \indent \indent \indent BUTE } \leftline{Hungarian
Academy of Sciences \indent \indent \indent \indent \indent
\indent \indent Department of Algebra} \leftline{Budapest Pf. 127
\indent \indent \indent \indent \indent \indent \indent and
\indent \indent \indent Budapest, Egry J. u. 1} \leftline{H-1364
Hungary \indent \indent \indent \indent \indent \indent \indent
\indent \indent \indent \indent \ \ \  H-1111 Hungary}
\leftline{e-mail: sagi@renyi.hu \indent \indent \indent \indent
\indent \indent \indent \indent \indent \indent \ e-mail:
sagi@math.bme.hu}

\end{document}